\documentclass[11pt]{article}

\usepackage{mathrsfs}

\usepackage{soul}

\usepackage{tikz}
\usetikzlibrary{calc}

\usepackage[margin=1in]{geometry}
\usepackage[utf8]{inputenc}
\usepackage{enumerate}
\usepackage{amsfonts}
\usepackage{amsmath}
\usepackage{amssymb}
\usepackage{hyperref}
\usepackage{amsthm}
\usepackage{xcolor}
\usepackage{tikz}
\usepackage{comment}
\usepackage{kotex}
\usepackage[numbers]{natbib}

\usepackage{colonequals}
\usepackage{theoremref}
\usepackage{enumitem}
\numberwithin{equation}{section}
\usepackage{appendix}

\newcommand{\e}{\mathbb E}

\newcommand{\p}{\mathbb P}
\newcommand{\1}{\mymathbb{1}}

\DeclareMathAlphabet{\mymathbb}{U}{BOONDOX-ds}{m}{n}
\newcommand{\eps}{\varepsilon}

\newcommand{\la}{\langle}
\newcommand{\ra}{\rangle}

\newcommand{\GSE}{\mathsf{GSE}}

\newcommand{\bal}{{\sf bal}}

\newtheorem{theorem}{Theorem}[section]

\newtheorem{proposition}[theorem]{Proposition}
\newtheorem{lemma}[theorem]{Lemma}

\newtheorem{definition}[theorem]{Definition}

\theoremstyle{definition} 
\newtheorem*{ack}{Acknowledgments}
\newtheorem{remark}[theorem]{Remark}

\begin{document}
\title{Color symmetry in the Potts spin glass at high temperature}
\date{}

 
\author{ Heejune Kim \thanks{ Email: kim01154@umn.edu.} }

\maketitle

\begin{abstract}


We show that color symmetry is preserved at high temperatures in the Potts spin glass model with $\kappa \ge 3$ colors. Our proof employs the second moment method applied to the balanced model with a suitable centering of the Hamiltonian, while incorporating results from the non-disordered Potts model \cite{EW90}. For $\kappa = 2$, we exploit the model's gauge symmetry to show that unbalanced configurations occur with exponentially small probability at all temperatures $\beta \in [0, \infty]$. 

\end{abstract}


\section{Introduction and main results}

The Potts spin glass is a $\kappa$-color generalization of the Sherrington–Kirkpatrick (SK) model.
For a system of $N$ spins, let $\Sigma_\kappa = [\kappa]^N$ denote the set of configurations. For any $\sigma \in \Sigma_\kappa$, the magnetization vector $d(\sigma) \in \mathcal{D}_\kappa \colonequals \{d \in [0,1]^\kappa : d_1+\dots+d_\kappa = 1\}$ is given by$$d(\sigma) = \bigl(N^{-1}|\{i \in [N] : \sigma_i = a\}|\bigr)_{a \in [\kappa]}.$$
We define the set of constrained configurations by $\Sigma^d_\kappa = \{ \sigma \in \Sigma_\kappa : d(\sigma) = d \}$ for $d\in\mathcal D_\kappa$.
In particular, a configuration is called balanced if $d(\sigma) = \kappa^{-1}\1$, and we let $\Sigma^\bal_\kappa$ denote the collection of all such configurations.
We identify each configuration $\sigma = (\sigma_1, \dots, \sigma_N)$ with a $\kappa \times N$ matrix whose columns are standard basis vectors $(e_a)_{a\in[\kappa]}$ in $\mathbb{R}^\kappa$.
By a slight abuse of notation,  we will use $\sigma_i = e_a$ and $\sigma_i = a$ interchangeably.
The Hamiltonian of the  Potts spin glass model is given by \begin{equation}\label{eq: Hamiltonian}
    H_N(\sigma)=\frac{1}{\sqrt{N}}\sum_{i,j \in [N]}g_{ij}\1\{\sigma_i=\sigma_j\}=\frac{1}{\sqrt{N}}\sum_{i,j \in [N]}g_{ij}\sigma_i^\intercal \sigma_j  \quad  \text{for} \quad \sigma \in \Sigma_\kappa,
\end{equation} where $(g_{ij})_{ i,j\in [N]}$ are i.i.d. standard normal random variables.
 Associated with \eqref{eq: Hamiltonian} are the ground state energy $\GSE_{N,\kappa} = N^{-1} \e \max_{\sigma\in\Sigma_\kappa} H_N(\sigma)$  and the free energy at inverse temperature $\beta \in [0, \infty)$,
\begin{equation*}
    F_{N,\beta} = \frac{1}{N} \e \log Z_{N,\beta}, \quad  \text{where}\quad Z_{N,\beta} \colonequals \sum_{\sigma \in \Sigma_\kappa} \exp(\beta H_N(\sigma))
\end{equation*} is called the partition function.
The corresponding  Gibbs measure $G_{N,\beta}$ is specified by  its probability mass function \begin{equation*}\label{eq: Gibbs measure}
    G_{N,\beta}(\sigma)= \frac{\exp ( \beta H_N(\sigma) )}{Z_{N,\beta}}  \quad \text{for} \quad \sigma \in \Sigma_\kappa,
\end{equation*} 
and we denote the average with respect to the product measure $G_{N,\beta}^{\otimes \infty}$ by $\la\cdot\ra_{N,\beta}.$
In the zero-temperature limit $\beta = \infty$, $G_{N,\infty}$ denotes the uniform distribution on the set of configurations maximizing $H_N(\sigma)$.

We similarly define the balanced ground state energy $\GSE^\bal_{N,\kappa}$, the balanced free energy $F^\bal_{N,\beta}$, and the balanced partition function $Z^\bal_{N,\beta}$ by restricting to the balanced configurations $\Sigma^\bal_\kappa$ and assuming $N\in\kappa \mathbb N$.
More generally,  the model  constrained at $\Sigma^d_\kappa$ can be defined for a magnetization vector $d \in \mathcal{D}_\kappa \cap \mathbb{Q}^\kappa$ and $N \in \operatorname{lcm}(d)\mathbb{N}$, where $\operatorname{lcm}(d)$ is the least common multiple of the denominators of $d$ in reduced form.
Henceforth, we assume that $N$ satisfies the necessary divisibility conditions whenever a constraint is imposed on the model.

Panchenko \cite{Pan18_Potts} established the Parisi formula for the limiting free energy  $\lim_{N\to\infty} F_{N,\beta}$ as a supremum of the Parisi formula for the model constrained around $\Sigma^d_\kappa$ for each  $d\in\mathcal D_\kappa$.
Building on this work, Bates--Sohn \cite[Theorem 1.7]{BS24} demonstrated that this constrained Parisi formula indeed coincides with the limiting constrained free energy $\lim_{N\to\infty}F_{N,\beta}^d$ which is uniformly continuous in $d \in \mathcal D_\kappa \cap \mathbb Q^\kappa$ (see \cite[Lemma E.4]{BS24}).
We can  summarize these findings as  \begin{equation}\label{eq: summary of Pan and BS}
    \lim_{N\to\infty} F_{N,\beta} =\max_{d\in\mathcal D_\kappa}\lim_{N\to\infty}F_{N,\beta}^d,
\end{equation}
where we extended $\lim_{N\to\infty}F^d_{N,\beta}$  to all $d \in \mathcal D_\kappa$ by uniform continuity.
In the special case of  balanced configurations,  Bates--Sohn  \cite[Theorems 1.3 and 1.8]{BS24} showed that the Parisi formula for the limiting balanced free energy  $\lim_{N\to\infty}F^\bal_{N,\beta}$ simplifies considerably in terms of its order parameter, resembling the SK model.
They further suggested that this limiting balanced free energy  always achieves the maximum in \eqref{eq: summary of Pan and BS}, that is, color symmetry holds at any $\beta\in [0,\infty)$.
We formalize this notion as follows, including the case of limiting ground state energies.
\begin{definition}[Color symmetry]
    We say that  the Potts spin glass model preserves color symmetry   at inverse temperature $\beta\in [0,\infty)$  if $\lim_{N\to\infty} F^\bal_{N,\beta} = \lim_{N\to\infty}F_{N,\beta}$, and at  zero-temperature $\beta=\infty$ if $\lim_{N\to\infty}\GSE^\bal_{N,\kappa}=\lim_{N\to\infty}\GSE_{N,\kappa}$\footnote{The  existence of the limits  $\lim_{N\to\infty}\GSE^\bal_{N,\kappa}$ and $\lim_{N\to\infty}\GSE_{N,\kappa}$ easily follows from that of the corresponding free energies. Indeed, from the elementary inequalities $\GSE_{N,\kappa} \le \beta^{-1}F_{N,\beta}\le \GSE_{N,\kappa}+\beta^{-1}\log \kappa$  and the existence of the limit $\lim_{N\to\infty}F_{N,\beta}$, we conclude that $\lim_{N\to\infty}\GSE_{N,\kappa}=\inf_{\beta>0}(\beta^{-1}\lim_{N\to\infty}F_{N,\beta})$. The balanced case is similar.}.
    Otherwise, the model breaks color symmetry at $\beta.$
\end{definition}
\begin{remark}
By the concentration of Gaussian measures, color symmetry breaking at $\beta$ is equivalent to the statement that the balanced configurations occur with exponentially small  probability under $\e G_{N,\beta}$ (see, e.g.,  \cite[Proposition 13.4.3]{Talagrand2013vol2}).
\end{remark}

Regarding the aforementioned prediction by Bates--Sohn \cite{BS24}, Mourrat \cite{Mou25} showed, contrary to the prediction, that color symmetry breaks  on an  $O(\sqrt{\kappa})$-window of $\beta$ centered at $2\kappa/(3\sqrt{\pi})$ whenever $\kappa\ge 58$. 
We expect that color symmetry breaks at zero temperature for all $\kappa \ge 3$ (see the first item in the open problems below),  as predicted in the physics literature  \cite{CPR12, SPR95, ES83_curious, ES83}.

In the present work, we establish color symmetry at high temperatures for all $\kappa\ge 3.$
To state our main result, we define a high-temperature threshold for color symmetry breaking, for $\kappa \ge 3$, as
\begin{equation}
    \beta_\kappa = \sqrt{\kappa(\kappa-1)\log(\kappa-1)} \cdot \min \Bigl\{ \frac{1}{\sqrt{\kappa-2}}, \, \frac{\sqrt{2}}{\kappa-2} \Bigr\} . \label{eq: high temperature threshold}
\end{equation}
For $\kappa\ge 4$, this minimum is achieved by the second term, which is obtained from the critical temperature of a corresponding non-disordered ferromagnetic Potts model in \cite[Theorem 2.1]{EW90}. 
Our main theorem characterizes the high-temperature behavior for all $\kappa \ge 3$ within the regime $\beta\in[0,\beta_\kappa).$
\begin{theorem}\label{thm: high temp}
    For any $\kappa\ge 3$ and  $\beta \in [0, \beta_\kappa)$, we have \[ \lim_{N\to\infty} \e F^\bal_{N,\beta}=\lim_{N\to\infty} \e F_{N,\beta} = \log \kappa +\frac{\beta^2 (\kappa-1)}{2\kappa^2}.\] 
    In particular,  color symmetry is preserved in this high temperature regime.
\end{theorem}
\begin{remark}
    The right-hand side  can be identified with the replica symmetric solution of the balanced model, where the order parameter is given by the constant path $\pi(t) = \kappa^{-2}\1\1^\intercal$ for all $0<t\le 1$ in the notation of \cite[Theorem 1.3]{BS24}. 
    Notably, this coincides with the predicted high-temperature value of the unconstrained limiting free energy in the physics literature \cite{SPR95, LS84}. 
\end{remark}

We complement Theorem~\ref{thm: high temp} with a result for   $\kappa=2$, showing that unbalanced configurations occur with exponentially small  probability for all $\beta \in [0,\infty]$.
Denote  the $\ell^\infty$-norm on  $\mathbb R^2$ by $\|\cdot\|_\infty$.
\begin{proposition}\label{prop: kappa=2}
    Assume $\kappa=2$. Then, for any $\beta \in [0,\infty]$ and $\eps >0$, \[ \e G_{N,\beta} \bigl( \|d(\sigma)-2^{-1}\1 \|_\infty \ge \eps \bigr)\le 2 e^{-\eps^2 N}.\]
    Consequently, color symmetry is preserved for all $\beta\in [0,\infty]$.
\end{proposition}
 

\paragraph{Open problems.} 
With the above  results, we list a few  related questions.
\begin{enumerate}
    
    \item Does color symmetry break at zero temperature for all $\kappa \ge 3$? 
    Noticing that $\GSE_{N,\kappa}^\bal\le \GSE_{N,\kappa}\le \GSE_{N,\kappa+1}$ due to the trivial inclusions $\Sigma^\bal_\kappa \subseteq \Sigma_\kappa\subseteq \Sigma_{\kappa+1}$, this would follow if one could show the following strict inequality, \[\lim_{N\to\infty}\GSE^\bal_{N,\kappa} > \lim_{N\to\infty}\GSE^\bal_{N,\kappa+1}\quad \text{for} \quad \kappa \ge 2,\] which we conjecture to be true.
   We refer the reader to Appendix~\ref{sec: color symmetry breaking at zero temp} for a proof of color symmetry breaking at zero temperature for the restricted range $\kappa \ge 56$.

    \item For $\kappa\ge3,$ is there a critical temperature  $\hat \beta_{\kappa}\in (0,\infty]$ of color symmetry breaking, where color symmetry is preserved for $\beta\le \hat \beta_{\kappa}$ and broken for $\beta>\hat \beta_{\kappa}$?  
    If so,  Theorem~\ref{thm: high temp} provides a lower bound  $\hat\beta_{\kappa} \ge \beta_\kappa$.

    \item For $\kappa\ge 3$, find the maximizers  $d\in\mathcal D_\kappa$, possibly depending on $\beta$, that achieves the supremum in \eqref{eq: summary of Pan and BS}.
    Is the maximizer unique modulo the permutation symmetry of colors?


\end{enumerate}

\subsection{Proof sketch}

For $\kappa \ge 3$, the balanced free energy is analyzed via the second moment method. The centering of the Hamiltonian is critical to this approach; without such an adjustment, the second moment method is inherently bound to fail, as demonstrated in Appendix~\ref{sec: necessity}. 
Most of the work lies in controlling the ratio \eqref{eq: exponential moment of balanced}, which compares the second moment of the centered and balanced partition function to the square of its first moment. To manage this, the ratio is decomposed into two distinct components: the first is addressed via a local expansion of the Kullback–Leibler divergence (see Lemma~\ref{lem: local KL divergence}), while the second is controlled by applying large deviation results from the non-disordered Potts model \cite{EW90} (see Lemma~\ref{lem: exponent gap in high termperature}). 
This framework allows us to establish the limiting balanced free energy as in Proposition~\ref{prop: high temperature of balanced}, and it turns out to coincide with the annealed limiting free energy of the unconstrained model from \cite{EW90}, thereby proving Theorem~\ref{thm: high temp}.

For $\kappa=2$,  the Potts spin glass model reduces to the SK model by a simple change of variable $\1\{\tau_i=\tau_j\}=2^{-1}(\tau_i\tau_j+1)$  for $\tau\in\{-1,+1\}^N$, where Jagannath--Sen \cite[Appendix C]{JS21} established color symmetry at zero temperature $\beta=\infty$.
Their method can be extended to  $\beta<\infty$, which together with some additional properties of the Parisi formula would imply Proposition~\ref{prop: kappa=2}.
Nevertheless, we present an elementary proof without the Parisi formula by directly controlling the multi-spin correlations via the gauge symmetry of the SK model, namely, the invariance of the distribution of the Hamiltonian under the map $\tau \mapsto (a_i\tau_i)_{i\in[N]}$ for any $a,\tau \in\{-1,+1\}^N$ (see Lemma~\ref{lem: multi-spin correlation}).


\paragraph{Related work.}

Chen \cite{Che25} demonstrated that color symmetry is preserved when a sufficiently strong antiferromagnetic interaction is integrated into the Hamiltonian. 
Relatedly, Chen \cite{Che24_Potts} established that this ``self-overlap correction'' simplifies the Parisi formula to the one equivalent to the balanced Potts spin glass \cite{BS24}, which was subsequently generalized to permutation-invariant models by Issa \cite{Iss24}.

The connection between the SK model and the Max-cut problem, initially explored in \cite{DMS17}, was generalized by Sen \cite{Sen18} who established that the balanced ground state energy of the Potts spin glass serves as the leading-order correction term for the Max $\kappa$-cut problem on sparse random graphs.
This relationship  was later extended to inhomogeneous sparse random graphs in \cite{JKS18}.

Finally, while gauge symmetry has long been known in the physics literature (cf. \cite[Chapter 4]{Nis01_book}), it has recently been rediscovered as a key component in the proof of the disorder chaos phenomenon in the Edwards–Anderson model \cite{Chatterjee} and diluted spin glass models \cite{CKS24}.





\begin{ack}
    The author thanks Wei-Kuo Chen and Arnab Sen for valuable comments.
    This work was partially supported by NSF grant
DMS-2246715.
\end{ack}

\section{Proof of Theorem~\ref{thm: high temp}}\label{sec: proof of main theorem}
Fix $\kappa\ge 3$ throughout this section.
As aforementioned in the proof sketch, we consider the centered Hamiltonian  \begin{equation}\label{eq: centered Hamiltonian}
    H_{N,\kappa}(\sigma) \colonequals \frac{1}{\sqrt{N}}\sum_{i,j\in[N]}g_{ij} (\sigma^\intercal_i \sigma_j - \kappa^{-1})=\frac{1}{\sqrt{N}}\sum_{i,j\in[N]}g_{ij} \sigma^\intercal_i P_\kappa \sigma_j, 
\end{equation} where \[P_\kappa\colonequals I_\kappa -\kappa^{-1}\1\1^\intercal.\]
One can check  that $P_\kappa^2=P_\kappa=P^\intercal_\kappa$ and  it is indeed a orthogonal projection matrix onto the orthogonal complement of the one-dimensional subspace generated by $\1$.
Denote the  partition function and the balanced partition function corresponding to the centered Hamiltonian \eqref{eq: centered Hamiltonian} by $Z_{N,\beta,\kappa}$ and $Z^\bal_{N,\beta,\kappa}$, respectively.
Since  $H_{N,\kappa}(\sigma)-H_N(\sigma)$ is a mean-zero constant  not depending on $\sigma$, we may use these centered Hamiltonian instead of  the original one \eqref{eq: Hamiltonian} to calculate the free energies, that is, \begin{equation}\label{eq: centering does not change the balanced free energy}
     F_{N,\beta}=\frac{1}{N}\e\log Z_{N,\beta,\kappa} \quad \text{and}\quad F^\bal_{N,\beta}= \frac{1}{N}\e\log Z^\bal_{N,\beta,\kappa}.
\end{equation}
Our main result of this section is that for the centered Hamiltonian \eqref{eq: centered Hamiltonian},  the limiting balanced free energy is equal to its annealed counterpart in a high-temperature regime.
\begin{proposition}\label{prop: high temperature of balanced}
    Let $\kappa\ge 3$. 
    For $ 0\le \beta<  (\kappa-2)^{-1/2}\sqrt{\kappa(\kappa-1)\log(\kappa-1)} $, we have \[ \lim_{N\to\infty} \frac{1}{N}\e\log Z^\bal_{N,\beta,\kappa} = \lim_{N\to\infty}\log \e Z^\bal_{N,\beta,\kappa}= \log \kappa +\frac{\beta^2(\kappa-1)}{2\kappa^2}.\]
\end{proposition}
Before proving  this, we demonstrate that Proposition~\ref{prop: high temperature of balanced} readily implies Theorem~\ref{thm: high temp}.

\begin{proof}[\bf Proof of Theorem~\ref{thm: high temp}]
Recall $\beta_\kappa$ in \eqref{eq: high temperature threshold} and let $\beta <\beta_\kappa$.
Notice that \begin{equation}
        \e Z_{N,\beta,\kappa} = \sum_{\sigma \in \Sigma_\kappa} e^{ (2N)^{-1}\beta^2 \sum_{i,j \in [N]} (\1\{\sigma_i=\sigma_j\}-\kappa^{-1})^2}=e^{(2\kappa^2)^{-1}\beta^2N}\sum_{\sigma \in \Sigma_\kappa} e^{ (2N)^{-1}\beta^2 a_\kappa \sum_{i,j \in [N]}\1\{\sigma_i=\sigma_j\}}, \label{eq: ferro Potts}
\end{equation} 
where we used the relation, for $a_\kappa \colonequals 1-2\kappa^{-1}$, 
\begin{equation*}
    (\1\{\sigma_i=\sigma_j\} - \kappa^{-1})^2= a_\kappa\1\{\sigma_i=\sigma_j\}+\kappa^{-2}. 
\end{equation*} 
The right-hand side of \eqref{eq: ferro Potts} corresponds to a non-disordered ferromagnetic Potts model at inverse temperature $\beta^2 a_\kappa$ considered in \cite{EW90}. 
Since $\beta$ satisfies \[\beta^2a_\kappa  < \frac{2(\kappa -1)}{\kappa-2}\log (\kappa -1), \; \text{that is,} \; 0\le \beta < \frac{1}{\kappa-2} \sqrt{2\kappa(\kappa-1)\log (\kappa-1)},\]  \cite[Theorem 2.1]{EW90} and \eqref{eq: ferro Potts} imply  \begin{equation*}
    \lim_{N\to\infty}\frac{1}{N}\log \e  Z_{N,\beta,\kappa}=\frac{\beta^2}{2\kappa^2}+\log \kappa +\frac{\beta^2a_\kappa}{2\kappa}= \log \kappa +\frac{\beta^2(\kappa-1)}{2\kappa^2}.
\end{equation*}
From Jensen's inequality and \eqref{eq: centering does not change the balanced free energy}, we have $ F_{N,\beta}\le N^{-1}\log \e Z_{N,\beta,\kappa}$, which  together with the last display show \begin{equation*}
    \lim_{N\to\infty} F_{N,\beta}\le \log \kappa +\frac{\beta^2(\kappa-1)}{2 \kappa^2}.
\end{equation*}
From a trivial inequality $ F^\bal_{N,\beta}\le  F_{N,\beta}$, Proposition~\ref{prop: high temperature of balanced} and \eqref{eq: centering does not change the balanced free energy} imply the announced result \[\lim_{N\to\infty} F^\bal_{N,\beta}=\lim_{N\to\infty} F_{N,\beta}=\log \kappa +\frac{\beta^2(\kappa-1)}{2 \kappa^2}.\]

\end{proof}

The rest of the section is devoted to the proof of Proposition~\ref{prop: high temperature of balanced}, for which we utilize the second moment method.
Define the $\kappa$-by-$\kappa$ overlap matrix by \[R(\sigma,\tau) = \frac{1}{N}\sigma \tau^\intercal \quad \text{for}\quad \sigma,\tau \in \Sigma_\kappa. \]
The entries of the overlap matrix is given by \begin{equation}\label{eq: entries of the overlap}
    NR_{ab}(\sigma,\tau)=(\sigma \tau^\intercal)_{ab}=\sum_{i\in[N]}\sigma_{a i}\tau_{bi}=\sum_{i\in [N]}\1\{\sigma_i=a, \tau_i=b\}, \quad a,b\in[\kappa].
\end{equation}
Let $\mathbb Z_{\ge 0} \colonequals \mathbb N\cup \{0\}$. 
For balanced configurations  $\sigma,\tau \in \Sigma^\bal_\kappa$, it can be seen from \eqref{eq: entries of the overlap} that the overlap matrix $R(\sigma,\tau)$ belongs to  the collection of admissible matrices\footnote{$\mathcal A^\bal_{\kappa}$ lies in a $(\kappa^2- 2\kappa+1)$-dimensional affine Euclidean subspace, and indeed, in a dilation of the Birkhoff polytope.} \begin{equation*}
    \mathcal A^\bal_{\kappa}\colonequals \Bigl\{(r_{ab})_{a,b \in [\kappa]}\in (N^{-1}\mathbb Z_{\ge 0})^{\kappa \otimes \kappa}: \sum_{a\in [\kappa]}r_{ab}=\sum_{b\in [\kappa]}r_{ab}=\kappa^{-1} \Bigr\}. \label{eq: admissible matrix}
\end{equation*}
Denote the unit vectors of $\mathbb R^N$ by $(e_i)_{i\in[N]}$.
By a slight abuse of notation, we also denote the unit vectors in  $\mathbb R^\kappa$ by $(e_a)_{a\in[\kappa]}$.
The overlap matrix is closely related to the covariance of the Hamiltonian \eqref{eq: Hamiltonian}.
Let us denote the Frobenius norm of a square matrix by $\|\cdot \|_F$.
\begin{lemma}\label{eq: covariance}
    For $\sigma,\tau \in \Sigma_{\kappa}$, we have \[\e  H_N(\sigma)  H_N(\tau)= N \|  R(\sigma,\tau) \|_F^2. \]
\end{lemma}
\begin{proof}

The covariance can be calculated as \begin{equation*}
    N\e (H_N(\sigma) H_N(\tau)) = \sum_{i,j \in [N]}e_i^\intercal \sigma^\intercal \sigma e_j e_i^\intercal \tau^\intercal \tau e_j
     =\sum_{i,j \in [N]}\sum_{a,b\in[\kappa]}(e_i^\intercal \sigma^\intercal e_a) (e_a^\intercal \sigma e_j) (e_i^\intercal \tau^\intercal e_b) (e_b^\intercal \tau e_j),
\end{equation*} where we used the resolution of identity $\sum_{a\in[\kappa]}e_a e_a^\intercal=I_\kappa.$
Notice the scalar quantities  $e_a^\intercal \sigma e_i =e_i^\intercal \sigma^\intercal e_a$ and $e^\intercal_b\tau e_j= e^\intercal_j \tau^\intercal e_b$.
We can then write the preceding display as \begin{equation*}
    \sum_{a,b\in[\kappa]} \sum_{i,j \in [N]}(e_a^\intercal \sigma e_i) (e_i^\intercal \tau^\intercal e_b) (e_a^\intercal \sigma e_j)  (e_j^\intercal \tau^\intercal e_b) = \sum_{a,b\in[\kappa]} e^\intercal_a\sigma \tau^\intercal e_b e^\intercal_a \sigma \tau^\intercal e_b
    =\sum_{a,b\in[\kappa]} e^\intercal_a\sigma \tau^\intercal e_b   e^\intercal_b \tau \sigma^\intercal e_a,
\end{equation*} again using the resolution of identity $\sum_{i\in[N]}e_ie^\intercal_i=I_N$.
Finally, noticing another resolution of identity $\sum_{b\in[\kappa]}e_be^\intercal_b=I_\kappa$, we arrive at the announced result \[ \sum_{a\in [\kappa]}e^\intercal_a\sigma \tau^\intercal \tau \sigma^\intercal e_a
    = \operatorname{Tr} (\sigma \tau^\intercal \tau \sigma^\intercal) =N^2 \operatorname{Tr} (R(\sigma,\tau) R(\sigma,\tau)^\intercal) =N^2\|R(\sigma,\tau)\|_F^2.\]

\end{proof}
A similar calculation can be done for the centered Hamiltonian \eqref{eq: centered Hamiltonian}.
\begin{lemma}\label{lem: centered covariance}
    For $\sigma,\tau \in \Sigma_{\kappa}$, we have \[\e  H_{N,\kappa}(\sigma)  H_{N,\kappa}(\tau)= N \| P_\kappa R(\sigma,\tau) P_\kappa\|_F^2. \]
    For balanced configurations $\sigma,\tau\in\Sigma^\bal_\kappa$, we have \begin{equation}
        P_{\kappa}R(\sigma,\tau)P_{\kappa}
    =R(\sigma,\tau)- \kappa^{-2}\1\1^\intercal. \label{eq: balanced overlap calculation}
    \end{equation}    
\end{lemma}
\begin{proof}
Since $P_\kappa^2=P_\kappa$ as a projection, we can write $H_{N,\kappa}(\sigma)= H_N(P_\kappa \sigma)$  for $\sigma\in \Sigma_\kappa$.
Then  Lemma~\ref{eq: covariance} implies $\e  H_{N,\kappa}(\sigma)  H_{N,\kappa}(\tau)= N^{-1}\| P_\kappa \sigma \tau^\intercal P_\kappa\|_F^2=N \| P_\kappa R(\sigma,\tau) P_\kappa\|_F^2$.
Next, notice that  $\sigma,\tau\in\Sigma^\bal_\kappa$ implies   $R(\sigma,\tau)\1=\kappa^{-1}\1$ and hence  \begin{equation*}
    P_{\kappa}R(\sigma,\tau)P_{\kappa}= R -\kappa^{-1} \1 \1^\intercal R(\sigma,\tau) - \kappa^{-1} R(\sigma,\tau)\1 \1^\intercal + \kappa^{-2}\1\1^\intercal R(\sigma,\tau)\1 \1^\intercal \nonumber
    =R(\sigma,\tau)- \kappa^{-2}\1\1^\intercal.
\end{equation*}  

\end{proof}

From Lemma~\ref{lem: centered covariance}, the variance $\e H_{N,\kappa}(\sigma)^2=N\kappa^{-2}\|P_\kappa\|_F^2$ is constant in $\sigma \in \Sigma^\bal_\kappa$ and, hence,
\begin{align}
    \e Z^\bal_{N,\beta,\kappa} &=\sum_{\sigma \in \Sigma^\bal_\kappa } e^{2^{-1}\beta^2\e H_{N,\kappa}(\sigma)^2}= e^{\beta^2 N \kappa^{-2}\|P_\kappa\|_F^2/2}\cdot |\Sigma_\kappa^\bal|  \label{eq: first moment rewrite} \quad \text{and} 
    \\ \e ( Z^\bal_{N,\beta,\kappa}) ^2 &= \sum_{\sigma,\tau \in \Sigma_\kappa^\bal}  e^{2^{-1}\beta^2 (H_{N,\kappa}(\sigma)+H_{N,\kappa}(\tau))^2} 
    = e^{\beta^2 N \kappa^{-2}\|P_\kappa\|_F^2}\sum_{\sigma,\tau \in \Sigma_\kappa^\bal}  e^{\beta^2 N\|P_\kappa R(\sigma,\tau)P_\kappa\|_F^2}.  \nonumber
\end{align}
The prefactors cancel to yield \begin{equation}\label{eq: exponential moment of balanced}
    \frac{\e ( Z^\bal_{N,\beta,\kappa}) ^2 }{(\e  Z^\bal_{N,\beta,\kappa}) ^2 } = \bigl\la \exp\bigl( \beta^2 N \|P_\kappa R(\sigma,\tau)P_\kappa\|_F^2\bigr) 
    \bigr\ra_\bal= \bigl\la \exp\bigl( \beta^2 N \|R(\sigma,\tau)-\kappa^{-2}\1\1^\intercal\|_F^2\bigr) 
    \bigr\ra_\bal,
\end{equation} where $\la \cdot \ra_\bal$ denotes the uniform average over $\Sigma_\kappa^\bal \times \Sigma_\kappa^\bal$, and the last equality  follows \eqref{eq: balanced overlap calculation}.
The next lemma shows that \eqref{eq: exponential moment of balanced} is bounded in $N$, which is the key component in the second moment method.
\begin{lemma}\label{lem: second moment ratio bounded}
Under the assumption of Proposition~\ref{prop: high temperature of balanced}, we have \[ \sup_{N\in \mathbb N} \frac{\e ( Z^\bal_{N,\beta,\kappa}) ^2 }{(\e  Z^\bal_{N,\beta,\kappa}) ^2 } <\infty.\]   
\end{lemma}
\begin{proof}
    Consider the uniform probability measure $\p_{\sigma,\tau}$ on $\Sigma_\kappa^\bal \times \Sigma_\kappa^\bal$.
    Condition on a fixed $\sigma \in \Sigma_\kappa^\bal$ and denote the conditional probability measure  by $\p_\tau$.
The matrix $R(\sigma,\tau)$ is then a random $\kappa$-by-$\kappa$ matrix with a fixed row and columns sums $1/\kappa$, and its conditional distribution is given by \begin{equation*}
    \p_\tau ( R(\sigma,\tau)= (r_{ab})_{a,b\in [\kappa]} )= \bigl(\frac{N!}{  ((N/\kappa)!)^{\kappa}} \bigr)^{-1} \prod_{a\in [\kappa]}\frac{(N/\kappa)!}{\prod_{b\in [\kappa]}(Nr_{ab})!} \quad \text{for} \quad (r_{ab})_{a,b\in [\kappa]} \in \mathcal A^\bal_{\kappa}. \label{eq: balanced table probability}
\end{equation*} 
Using the Stirling's formula $\log n!= n\log n -n +2^{-1}\log n+O(1)$ for $n\in\mathbb N$, we obtain    \begin{equation}\label{eq: LDP}
    \log \p_\tau ( R(\sigma,\tau) = (r_{ab})_{a,b\in [\kappa]} )= - N\sum_{a,b\in[\kappa]} r_{ab} \log (\kappa^2 r_{ab}) -\frac{(\kappa-1)^2}{2} \log N -\frac{1}{2}\sum_{(a,b) \in I_r}\log r_{ab}   +O(1),
\end{equation}  where  \[I_r\colonequals \{(a,b)\in[\kappa]\times [\kappa]: r_{ab}\ne 0 \}.\]
Note that this holds uniformly in $\sigma.$
We can identify the first term on the right-hand side \eqref{eq: LDP} (without the prefactor $N$) as a relative entropy (Kullback--Leibler divergence) with respect to the uniform measure on $[\kappa]\times [\kappa]$, which we write as $D(r\|\kappa^{-2}\1\1^\intercal).$
Combining \eqref{eq: exponential moment of balanced} and \eqref{eq: LDP}, we see that 
\begin{equation} \label{eq: second moment ratio}
\begin{aligned}
   \frac{\e ( Z^\bal_{N,\beta,\kappa}) ^2 }{(\e  Z^\bal_{N,\beta,\kappa}) ^2 } 
=  \sum_{r\in\mathcal A^\bal_{\kappa}} N^{-(\kappa-1)^2/2} &\exp\Bigl( -N\bigl( D(r\|\kappa^{-2}\1\1^\intercal)-\beta^2\|r-\kappa^{-2}\1\1^\intercal\|_F^2\bigr) 
\\&\qquad -\frac{1}{2}\sum_{a,b\in I_r}\log r_{ab}   +O(1)\Bigr)
\end{aligned}
\end{equation}
The following is a local expansion of the Kullback--Leibler divergence whose reference measure  has strictly positive mass, and we defer its proof to Appendix~\ref{sec: deferred proofs} (see also \cite[Theorem 4.1]{Csi04_book}).
\begin{lemma}\label{lem: local KL divergence}
    Let $n\in\mathbb N$.  
    Let $p\colonequals (p_i)_{i\in [n]}$ and $q\colonequals(q_i)_{i\in[n]}$ be two discrete probability distributions on $[n]$.
    Suppose $\min_{1\le i\le n} q_i>0$ and $\max_{i\in[n]}|p-q|\le 2^{-1}\min_{i\in[n]}q_i.$
   We have
     \begin{equation*}
    \Bigl| D(p\|q) - \frac{1}{2} \sum_{i\in [n]}\frac{(p_i-q_i)^2}{q_i} \Bigr|\le  5 \bigl(\min_{i\in [n]} q_i\bigr)^{-2} \sum_{i\in[n]}|p_i-q_i|^3.
\end{equation*} 
\end{lemma}

Fix $\eps>0.$
From Lemma~\ref{lem: local KL divergence}, there exists $\delta'>0$ such that if $\|r-\kappa^{-2}\1\1^\intercal\|_F^2<\delta'$, \begin{equation}\label{eq: local KL divergence}
    D(r\|\kappa^{-2}\1\1^\intercal)  \ge (1-\eps) \frac{\kappa^2}{2}\| r-\kappa^{-2}\1\1^\intercal\|_F^2.
\end{equation}
Moreover, there exists $K=K(\kappa)<\infty$ such that for a small enough $\delta''=\delta''(\kappa)>0$, we have \begin{equation}\label{eq: good around uniform}
    \sum_{a,b \in [\kappa]}|\log r_{ab}|\le K \quad \text{whenever} \quad  \| r-\kappa^{-2}\1\1^\intercal\|_F^2< \delta''.
\end{equation} 
Take $\delta =2^{-1}\min\{\delta',\delta''\}$.
  From \eqref{eq: local KL divergence} and \eqref{eq: good around uniform},  we can bound \eqref{eq: second moment ratio} by separating it into two parts:
\begin{align}
& \sum_{\substack{r\in\mathcal A^\bal_{\kappa}, \\\|r-\kappa^{-2}\1\1^\intercal\|_F^2\le \delta}} N^{-(\kappa-1)^2/2} \exp\Bigl( - \frac{(1-\eps)\kappa^2- 2\beta^2}{2} N \| r-\kappa^{-2}\1\1^\intercal\|_F^2 
+K   +O(1)\Bigr) \label{eq: centered balanced upper bound}
\\&+ \sum_{\substack{r\in\mathcal A^\bal_{\kappa},\\\|r-\kappa^{-2}\1\1^\intercal\|_F^2\ge \delta}} N^{-(\kappa-1)^2/2} \exp\Bigl( -N\bigl( D(r\|\kappa^{-2}\1\1^\intercal)- \beta^2\|r-\kappa^{-2}\1\1^\intercal\|_F^2\bigr)
-\frac{1}{2}\sum_{a,b\in I_r}\log r_{ab}   +O(1)\Bigr).  \label{eq: centered balanced upper bound: 2}
\end{align} 
Moreover, we have the following volume estimate, for $l \in [N]$,   \begin{equation}
    \Bigl|\Bigl\{ (r_{ab})_{a,b\in [\kappa]}\in \mathcal A^\bal_{\kappa}:  \frac{l-1}{N} \le \| r-\kappa^{-2}\1\1^\intercal\|_F^2< \frac{l}{N}\Bigr\} \Bigr| \le C'   ( lN )^{(\kappa-1)^2/2}  \label{eq: volume estimate for balanced}
\end{equation}
for some constant $C'>0$ depending only on $\kappa$. 
To see this, note that any admissible matrix in \eqref{eq: volume estimate for balanced} is uniquely determined by its leading principal submatrix of order $\kappa - 1$ and  each entry in this submatrix can take at most $O(1) \sqrt{lN}$ many values.
As a result, \eqref{eq: centered balanced upper bound} is  bounded  by
\begin{equation}
     O(1)\sum_{1\le l\le \delta N}   l^{(\kappa-1)^2/2}\exp \Bigl( - \frac{(1-\eps)\kappa^2- 2\beta^2}{2} (l-1) +K  \Bigr). \label{eq: centered balanced upper bound: 3}
\end{equation}
If  $2\beta^2<(1-\eps)\kappa^2$, then  \eqref{eq: centered balanced upper bound: 3} is uniformly bounded in $N$ as it is dominated by a convergent series $\sum_{l\ge 1}l^{(\kappa-1)^2/2}e^{-cl}<\infty$ for some $c>0$.

It remains to  estimate \eqref{eq: centered balanced upper bound: 2}.
To this end,  note the following application of  \cite{EW90}, whose proof is deferred to Appendix~\ref{sec: deferred proofs}.
\begin{lemma}\label{lem: exponent gap in high termperature}
    For any $\delta>0$, there exists $\eta>0$ such that if \[\beta^2<\frac{\kappa (\kappa-1)\log(\kappa-1)}{\kappa-2},\] then \[\inf_{\substack{r\in\mathcal A^\bal_{\kappa}, \\\|r-\kappa^{-2}\1\1^\intercal\|_F^2\ge \delta}} \bigl( D(r\|\kappa^{-2}\1\1^\intercal)- \beta^2\|r-\kappa^{-2}\1\1^\intercal\|_F^2\bigr)>\eta.\]
\end{lemma}

If $\beta^2<\kappa(\kappa-1)\log(\kappa-1)/(\kappa-2)$,  Lemma~\ref{lem: exponent gap in high termperature} guarantees an $\eta>0$ such that \eqref{eq: centered balanced upper bound: 2} is bounded by, in view of \eqref{eq: volume estimate for balanced},
\begin{equation}
     O(1)\sum_{\delta N< l\le N}  l^{(\kappa-1)^2/2}\exp \bigl( -\eta N +\kappa^2\log N  \bigr).\label{eq: centered balanced upper bound: 4}
\end{equation} 
The limit superior of \eqref{eq: centered balanced upper bound: 4} in $N$ is zero as the sum over $l$ is at most polynomial in $N$. 
Combining the previous estimates \eqref{eq: centered balanced upper bound: 3} and \eqref{eq: centered balanced upper bound: 4}, since $\eps>0$ is arbitrary, we conclude that \[ \sup_{N\in \mathbb N} \frac{\e ( Z^\bal_{N,\beta,\kappa}) ^2 }{(\e  Z^\bal_{N,\beta,\kappa}) ^2 } <\infty\] whenever \begin{equation*}
     \beta^2< \min \Bigl\{ \frac{\kappa^2}{2},\frac{\kappa (\kappa-1)\log(\kappa-1)}{\kappa-2}  \Bigr\}=\frac{\kappa (\kappa-1)\log(\kappa-1)}{\kappa-2}.
\end{equation*} 
This finishes the proof of Lemma~\ref{lem: second moment ratio bounded}.
\end{proof}

We are now in a position to complete the proof of Proposition \ref{prop: high temperature of balanced}.
\begin{proof}[\bf Proof of Proposition~\ref{prop: high temperature of balanced}]
Fix $N\in\mathbb N$.
Lemma~\ref{lem: second moment ratio bounded} and the Paley--Zygmund inequality show that \begin{equation*}
   \p( Z^\bal_{N,\beta,\kappa } \ge 2^{-1}\e  Z^\bal_{N,\beta,\kappa})>c>0
\end{equation*} for some constant $c>0$ depending only on $\kappa$ and $\beta.$
On this event of positive probability, we have \[\frac{1}{N}\log Z^\bal_{N,\beta,\kappa } \ge -\frac{\log 2}{N}+ \frac{1}{N}\log \e  Z^\bal_{N,\beta,\kappa}. \]
On the other hand, from the concentration of Gaussian measures (see \cite[Theorem 1.2]{panchenkobook}), it holds that $N^{-1}\log  Z^\bal_{N,\beta,\kappa } \to  N^{-1} \e \log Z^\bal_{N,\beta,\kappa}$ in probability as $N\to\infty.$
Therefore, given $\eps>0$, there exists an event of positive probability such that for sufficiently large $N,$ \begin{equation}\label{eq: on a good event}
    \frac{1}{N}\log \e  Z^\bal_{N,\beta,\kappa}\ge \frac{1}{N}\e \log  Z^\bal_{N,\beta,\kappa} \ge -\eps -\frac{\log 2}{N} +\frac{1}{N}\log \e Z^\bal_{N,\beta,\kappa},
\end{equation} where the first inequality is due to Jensen.
Note that $\|P_\kappa\|_F^2=\kappa^{-2}(\kappa-1)$ and $\lim_{N\to\infty}N^{-1}\log |\Sigma^\bal_\kappa|=\log \kappa $.
 From \eqref{eq: first moment rewrite},  we can then calculate the limit of the last term as  \[\lim_{N\to\infty} \frac{1}{N}\log \e  Z^\bal_{N,\beta,\kappa} =\log \kappa +\frac{\beta^2}{2}\frac{\kappa-1}{\kappa^2}.\] 
Since every quantity in  \eqref{eq: on a good event} is deterministic, taking $N\to\infty$ and then $\eps\to0$ shows that \[\lim_{N\to\infty}\frac{1}{N} \e \log  Z^\bal_{N,\beta,\kappa}=\lim_{N\to\infty} \frac{1}{N}\log \e  Z^\bal_{N,\beta,\kappa}=\log \kappa +\frac{\beta^2}{2}\frac{\kappa-1}{\kappa^2}.\] 
This finishes the proof.

\end{proof}

\section{Proof of Proposition~\ref{prop: kappa=2}}
This section is devoted to the proof of color symmetry of the Potts model with two colors $\kappa=2$.
As mentioned in the proof sketch, our main tool will be the gauge symmetry of the SK model.
We first identify the SK model with the Potts model for $\kappa=2$ as follows.
For $\sigma\in [2]^N$, consider the change of variable \[\tau_j= 2\1\{\sigma_{j}=1\}-1 \quad \text{for} \quad j\in [N].\]  
Note that $\tau \in \{-1,+1\}^N$, and it is easy to check $\tau_i\tau_j =2 \1\{\sigma_i=\sigma_j\}-1.$
The SK Hamiltonian is then obtained by this change of variable \begin{equation}\label{eq: usual SK hamiltonian}
         H_N^{\text{SK}}(\tau)\colonequals \frac{1}{\sqrt{N}}\sum_{i,j \in  [N]}g_{ij}\tau_i\tau_j  = 2 H_N(\sigma)- \frac{1}{\sqrt{N}}\sum_{i,j\in  [N]}g_{ij}.
    \end{equation}
   For $\beta \in [0,\infty]$, denote the corresponding Gibbs measure  by $G_{N,\beta}^{\text{SK}}$, and its Gibbs average by $\la \cdot\ra_{N,\beta}^{\text{SK}}$.
    Since the last term  in \eqref{eq: usual SK hamiltonian}   does not affect the Gibbs measure, we have \begin{equation}\label{eq: equivalent Gibbs measure}
        G_{N,\beta}^{\text{SK}}=G_{N,2\beta} \quad \text{and}\quad \la\cdot\ra_{N,\beta}^{\text{SK}}=\la\cdot\ra_{N,2\beta}.
    \end{equation}

For a finite sequence  $(i_1,\dots i_m)$  in $[N]$, define the degree of an index $l$ as $d_l=| \{l' \in [m]: i_{l'}=i_l\}|$.
Gauge symmetry of the SK model controls the multi-spin correlation of the Potts model as follows.
\begin{lemma}\label{lem: multi-spin correlation}
    Let $\beta \in[0,\infty]$ and $N\in\mathbb N$. 
    For any sequence  $(i_1,\dots i_m)$  in $[N]$, if there exists $ l_0 \in [m]$ such that $d_{l_0}$ is odd, then $\e  \la \prod_{l=1}^m  \tau_{i_l}\ra_{N,\beta}=0$.
\end{lemma}
\begin{proof}
    It is a standard fact that $G_{N,\beta}\stackrel{d}{\to}G_{N,\infty}$ as $\beta\to\infty$.
    Therefore, it suffices to consider the case $\beta<\infty$.
    Using \eqref{eq: equivalent Gibbs measure}, we may  rewrite the expectation as \begin{align}\label{eq: under the new measure}
    \e \Bigl \la \prod_{l=1}^m  \tau_{i_l}\Bigr \ra_{N,\beta}&= \e \Big \la \prod_{l=1}^m  \tau_{i_l}\Big \ra_{N,\beta/2}^{\text{SK}}
    \\&= \e\bigl( \sum_{\tau\in \{\pm 1\}^N}\exp(2^{-1}\beta  H^{\text{SK}}_N(\tau))\bigr)^{-1}\sum_{\tau \in \{\pm 1\}^N}\prod_{l=1}^m  \tau_{i_l} \cdot \exp(2^{-1}\beta  H^{\text{SK}}_N(\tau)). \nonumber
\end{align}
Consider the bijection $\tau \mapsto \hat \tau$ where $\hat \tau_{i_{l_0}}=-\tau_{i_{l_0}}$ and $\hat \tau_j =\tau_j$ for all $j\in [N]\setminus \{i_{l_0}\}$.
Since this map is a bijection on $\{-1,+1\}^N$,  the preceding display is equal to \begin{align*}
    &\e\Bigl( \sum_{\tau\in \{\pm 1\}^N}\exp(2^{-1}\beta  H^{\text{SK}}_N(\hat \tau))\Bigr)^{-1}\sum_{\tau \in \{\pm 1\}^N}\prod_{l=1}^m  \hat \tau_{i_l} \cdot \exp(2^{-1}\beta  H^{\text{SK}}_N(\hat \tau))
    \\&= \e\Bigl( \sum_{\tau\in \{\pm 1\}^N}\exp(2^{-1}\beta  \widetilde H^{\text{SK}}_N(\tau))\Bigr)^{-1}\sum_{\tau \in \{\pm 1\}^N} \Bigl((-1)^{d_{l_0}}\prod_{l=1}^m  \tau_{i_l}\Bigr) \cdot \exp(2^{-1}\beta  \widetilde H^{\text{SK}}_N(\tau)),
\end{align*} where $\sqrt{N} \widetilde H^{\text{SK}}_N (\tau)\colonequals -\sum_{l\in [N]\setminus\{i_{l_0}\}}(g_{i_{l_0},l}+ g_{l,i_{l_0}})\tau_{i_{l_0}}\tau_l + \sum_{l,l'\in [N]\setminus\{i_{l_0}\}}g_{ll'}\tau_l\tau_{l'}+ g_{i_{l_0},i_{l_0}}$.
Notice that we have the equality in distribution of the Gaussian processes, $(\widetilde H^{\text{SK}}_N (\tau))_{\tau \in \{\pm 1\}^N} \stackrel{d}{=} (H^{\text{SK}}_N (\tau))_{\tau \in\{\pm 1\}^N}$, 
 since $\widetilde H^{\mathrm{SK}}_N$ is obtained from $H^{\mathrm{SK}}_N$ by flipping the signs of the independent centered Gaussian couplings incident to $i_{l_0}$, which leaves the joint law invariant.
Hence, the last display is unchanged if we replace the process $\widetilde H^{\text{SK}}_N$ by $ H^{\text{SK}}_N$.
Combining this with \eqref{eq: equivalent Gibbs measure} and  the assumption that $d_{l_0}$ is odd, we conclude that \eqref{eq: under the new measure} is equal to \[-\e \Bigl\la \prod_{l=1}^m\tau_{i_l}\Bigr\ra_{N,\beta},\] which implies that it is indeed zero, thereby finishing the proof. 
\end{proof}

With the preceding lemma, we bound the  moments of the centered magnetization vector.
\begin{lemma}\label{lem: magnetization moment calculation}
    Let  $\beta\in [0,\infty]$.
    For any even $m \ge 0$, we have \[ \e \Bigl\la \Bigl(\frac{1}{N}\sum_{i\in [N]}\1\{\sigma_i=1\}-\frac{1}{2}\Bigr)^m \Bigr\ra_{N,\beta} \le \frac{m!}{2^m (m/2)! }\cdot \frac{1}{N^{m/2}}.\]
    For odd $m$, the left-hand side of the preceding display is zero.
\end{lemma}

\begin{proof}
Fix $m\ge 1.$
We write     \begin{align}
    &2^m\e \Bigl\la \Bigl(\frac{1}{N}\sum_{i\in [N]}\1\{\sigma_i=1\}-\frac{1}{2}\Bigr)^m \Bigr\ra_{N,\beta} = \e \Bigl\la \Bigl(\frac{1}{N}\sum_{i\in [N]} (2\1\{\sigma_i=1\}-1)\Bigr)^m \Bigr\ra_{N,\beta} \nonumber
    \\&= \frac{1}{N^m} \sum_{i_1,\dots,i_m \in [N]}\e \Bigl \la \prod_{l=1}^m  \tau_{i_l}\Bigr \ra_{N,\beta}. \label{eq: rewrite kappa=2}
\end{align}
Lemma~\ref{lem: multi-spin correlation} readily implies the desired result in the case of odd $m$, as there exists an index $l$ such that $d_{i_l}$ is odd.
It remains to check the case of even $m$.
Again by Lemma~\ref{lem: multi-spin correlation}, the only contributing terms in the sum \eqref{eq: rewrite kappa=2} are the ones such that all $d_{l}$'s are even for $1\le l\le m$, and each contribution is equal to $1$.
Therefore, we only need to count the number of ways to form a sequence $(i_1,\dots,i_m)$ in $[N]$ such that all $d_{l}$'s are even.
An elementary counting shows that it is bounded by    \[ \sum_{\substack{x_1,\dots, x_N \in \mathbb Z_{\ge 0},\\\sum_{i=1}^N x_i=m/2}} \frac{m!}{\prod_{1\le i\le N}(2x_i)!} \le  \frac{m!}{(m/2)!} \sum_{\substack{x_1,\dots, x_N \in \mathbb Z_{\ge 0},\\\sum_{i=1}^N x_i=m/2}}\frac{(m/2)!}{\prod_{1\le i\le N}x_i!}  = \frac{m!}{(m/2) !} N^{m/2}.\] Here, we used $(2x_i)!\ge x_i!$ and the multinomial theorem. 
We obtain the desired result by combining this with the prefactor $N^{-m}$ in \eqref{eq: rewrite kappa=2}.

\end{proof}
An immediate corollary is a control of the moment-generating function, from which we deduce the proof of Proposition~\ref{prop: kappa=2}.
\begin{lemma}\label{lem: exponential bound}
    Let $\beta \in [0,\infty]$. 
    For any $\lambda>0$, we have \[\e \Bigl\la \exp \Bigl( \lambda\bigl( \frac{1}{N}\sum_{i\in [N]}\1\{\sigma_i=1\}-\frac{1}{2}\bigr) \Bigr)\Big\ra_{N,\beta} \le e^{\lambda^2/(4N)}.\]
\end{lemma}
\begin{proof}
    Denote $X_N= N^{-1}\sum_{i\in [N]}\1\{\sigma_i=1\}-2^{-1}$.
    Let $\lambda >0$.
    From the Taylor expansion $e^x=\sum_{m\ge 0} x^m/m!$,  Lemma~\ref{lem: magnetization moment calculation} shows that \begin{align*}
        \e \la e^{\lambda X_N}\ra_{N,\beta} = \sum_{\text{even }m} \lambda^m \frac{\e\la X_N^m \ra_{N,\beta}}{m!} \le \sum_{\text{even }m} \frac{\lambda^m}{2^m (m/2)! }\cdot \frac{1}{N^{m/2}}= e^{\lambda^2/(4N)}.
    \end{align*} 
\end{proof}


\begin{proof}[\bf Proof of Proposition~\ref{prop: kappa=2}]
By a  Chernoff bound, Lemma~\ref{lem: exponential bound} yields \begin{equation}\label{eq: exponential concentration}
        \e G_{N,\beta} \Bigl(\,  \Bigl|\frac{1}{N}\sum_{i\in [N]}\1\{\sigma_i=1\}-\frac{1}{2} \Bigr| >\eps \Bigr) \le 2e^{-\eps^2N }.
    \end{equation}
It remains to establish color symmetry.
Fix $\eps>0$ and some $d\in\mathcal D_2 \cap \mathbb Q^2$ satisfying  $\|d-d_\bal\|_\infty > \eps$, which is equivalent to $|d_1-2^{-1} | >\eps$.
The proof is now finished  from Jensen's inequality and \eqref{eq: exponential concentration}: \begin{equation*}
 \e F^d_{N,\beta}- \e F_{N,\beta}= \frac{1}{N}\e \log G_{N,\beta} ( d(\sigma)=d) 
 \le \frac{1}{N} \log \e G_{N,\beta} ( |d_1-2^{-1} | >\eps)\le -\eps^2+o(1).
\end{equation*}
\end{proof}

\appendix

\section{Necessity of Hamiltonian centering}\label{sec: necessity}


In this section, we show that a suitable centering is required in the Hamiltonian for the second moment method to be effective in the proof of Theorem~\ref{thm: high temp}. Define the annealed Gibbs measure  at inverse temperature $\beta>0$ with the probability mass function \begin{equation*}\label{eq: annealed Gibbs measure}
     G_{\mathrm{ann},N,\beta}(\sigma) = \frac{\e \exp( \beta H_N(\sigma))}{\e Z_{N,\beta}}= \frac{\exp\bigl((2N)^{-1}\beta^2 \sum_{i,j\in[N]}\1\{\sigma_i=\sigma_j\}\bigr) }{\sum_{\sigma\in\Sigma_\kappa}\exp\bigl((2N)^{-1}\beta^2 \sum_{i,j\in[N]}\1\{\sigma_i=\sigma_j\}\bigr) } \quad \text{for}\quad \sigma \in \Sigma_\kappa,
 \end{equation*} 
 and denote the corresponding annealed Gibbs average by $\la\cdot\ra_{\mathrm{ann},N,\beta}$.
The following is in  a sharp contrast with Lemma~\ref{lem: second moment ratio bounded}, which justifies the consideration of the centered Hamiltonian \eqref{eq: centered Hamiltonian}.
\begin{lemma}
For any $\beta\in (0,\infty)$,
    \[\lim_{N\to\infty}\frac{\e Z_{N,\beta}^2}{ (\e Z_{N,\beta})^2} =\lim_{N\to\infty}\frac{\e (Z^{\bal}_{N,\beta})^2}{ (\e Z^\bal_{N,\beta})^2}= \infty. \]
\end{lemma}
\begin{proof}
From the identity
\begin{equation*}
    Z_{N,\beta}=\sum_{\sigma \in \Sigma_\kappa} e^{\beta H_N(\sigma)} = \sum_{\sigma \in \Sigma_\kappa}\prod_{i,j\in[N]} e^{\beta g_{ij} \1\{\sigma_i=\sigma_j\}/\sqrt{N}},
\end{equation*} 
 we obtain
 \begin{equation} \label{eq: moments}
\begin{aligned}
    (\e Z_{N,\beta})^2&=\sum_{\sigma,\tau \in \Sigma_\kappa }\prod_{i,j\in [N]} e^{\beta^2 (\1\{\sigma_i=\sigma_j\}+ \1\{\tau_i=\tau_j\})/(2N)} \quad \text{and}
    \\ \e Z_{N,\beta}^2 &= \sum_{\sigma,\tau \in \Sigma_\kappa}\prod_{i,j\in [N ]} e^{\beta^2 ( \1\{\sigma_i=\sigma_j\}+\1\{\tau_i=\tau_j\} )^2/(2N)}.    
\end{aligned} 
\end{equation}
By the Jensen's inequality and the symmetry among the sites, \begin{equation*}
    \begin{aligned}
    \frac{\e Z_{N,\beta}^2}{ (\e Z_{N,\beta})^2} = \la e^{\beta^2 \sum_{i,j}\1\{\sigma_i=\sigma_j, \tau_i=\tau_j\}/N}\ra_{\mathrm{ann},N,\beta} &\ge \exp\bigl( \beta^2  (N-1)G_{\mathrm{ann},N,\beta }(\sigma_1=\sigma_2 )^2\bigr) 
    \\&\ge \exp\bigl( \beta^2  (N-1)\kappa^{-2}\bigr),
\end{aligned}
\end{equation*}
where the last inequality follows from the coupling of the non-disordered Potts Gibbs measure and the random cluster measure; see  \cite[Theorem 1.16]{Gri09}.
This implies the desired divergence for the unconstrained model upon sending $N\to\infty$.
On the other hand, if we restrict to the balanced configurations  $\Sigma^\bal_\kappa$ in \eqref{eq: moments}, we have an analogous inequality \[\frac{\e (Z^\bal_{N,\beta})^2}{ (\e Z^\bal_{N,\beta})^2}=  \la e^{\beta^2 \sum_{i,j}\1\{\sigma_i=\sigma_j, \tau_i=\tau_j\}/N}\ra_{\bal}\ge \exp\bigl( \beta^2  (N-1)\p^\bal_{\operatorname{unif}}(\sigma_1=\sigma_2 )^2\bigr),\] where $\p^\bal_{\operatorname{unif}}$ and $\la\cdot \ra_\bal$ denote the uniform probability measure on $\Sigma^\bal_\kappa$ and uniform average on $\Sigma^\bal_\kappa \times \Sigma^\bal_\kappa$, respectively.
Since $\p^\bal_{\operatorname{unif}}(\sigma_2 =a \mid \sigma_1=a)= (N-1)^{-1}(\kappa^{-1}N-1)$ for any $a\in[\kappa]$, we have \[\p^\bal_{\operatorname{unif}}(\sigma_1=\sigma_2 )= \sum_{a\in[\kappa]}\p^\bal_{\operatorname{unif}}(\sigma_1=a) \p^\bal_{\operatorname{unif}}(\sigma_2 =a \mid \sigma_1=a)= \frac{N-\kappa}{(N-1)\kappa }.   \]
Combining the last two displays yield a diverging lower bound of the ratio for the balanced model.
This finishes the proof.
\end{proof}

\section{Deferred proofs}\label{sec: deferred proofs}
  We collect  deferred proofs of Section~\ref{sec: proof of main theorem}.
  \begin{proof}[\bf Proof of Lemma~\ref{lem: exponent gap in high termperature}]
    Using $\sum_{a,b}r_{ab}=1$, one can check \begin{equation}
        \|r-\kappa^{-2}\1\1^\intercal\|_F^2=\|r\|_F^2-\|\kappa^{-2}\1\1^\intercal\|_F^2.\label{eq: miracle of uniform}
    \end{equation} Therefore, our objective function is equivalent to   \begin{equation}
        D(r\|\kappa^{-2}\1\1^\intercal)- \beta^2 \|r\|_F^2 = \frac{1}{\kappa}\sum_{a\in[\kappa]}\Bigl(\sum_{b\in[\kappa]}v_{ab}\log (\kappa v_{ab}) -\frac{\beta^2}{\kappa}\sum_{b\in[\kappa]}v_{ab}^2\Bigr), \label{eq: equivalent objective function}
    \end{equation}  where we employed a change of variable $v_{ab}\colonequals \kappa r_{ab}$.
 
     Fix $a\in[\kappa]$ and denote each row sum by \[S_a(v_{a,\square}) \colonequals \sum_{b\in[\kappa]}v_{ab}\log (\kappa v_{ab}) -\kappa^{-1}\beta^2\sum_{b\in[\kappa]}v_{ab}^2,\]
     where $v_{a,\square}\colonequals (v_{ab})_{b\in[\kappa]}$.
     Note that $\sum_{b\in[\kappa]}v_{ab}=1$, so we can precisely identify $S_a$  with the negative of \cite[Equation (2.5)]{EW90}.
    By  \cite[Theorem 2.1-(b)]{EW90}, if \[\beta^2<\frac{\kappa (\kappa-1)\log(\kappa-1)}{\kappa-2},\] then the uniform distribution $v_{a,\square}=\kappa^{-1}\1$ is the unique minimizer of $S_a$.
    Noticing $S_a(\kappa^{-1}\1)=-\kappa^{-2}\beta^2$,
    we deduce that for any $\eps>0$, there exists $\eta>0$ such that \begin{equation}
        \inf_{v_{a,\square}\in \mathcal C_\eps} S_a(v_{a,\square})\ge -\frac{\beta^2}{\kappa^2}+\eta, \label{eq: consequence of uniqueness of minimizer}
    \end{equation}  where \[\mathcal C_\eps \colonequals \Bigl\{x \in [0,1]^\kappa:  \sum_{b\in[\kappa]}x_{b}=1, \, \|x -\kappa^{-1}\1\|_2^2 \ge \eps \Bigr\}.\]
    (We understand $\eta=0$ if $\eps=0$.)
    Indeed, this follows from the compactness of $\mathcal C_\eps$ and the uniqueness of the minimizer of $S_a.$

    Now, let $\delta>0$ and suppose $\|r-\kappa^{-2}\1\1^\intercal\|_F^2\ge \delta$ for some $r\in\mathcal A^\bal_{\kappa}$.
    The pigeonhole principle shows that there exists $\hat a\in [\kappa]$ such that $\sum_{b\in[\kappa]}(r_{\hat{a}b}-\kappa^{-2})^2\ge \kappa^{-1}\delta,$ i.e.,  \[\sum_{b\in[\kappa]}(v_{\hat{a} b}-\kappa^{-1})^2\ge \kappa \delta,\] which together with \eqref{eq: consequence of uniqueness of minimizer} guarantees that, for some $\eta>0$, \begin{equation*}
        S_{\hat{a}}(v_{\hat{a},\square})\ge -\frac{\beta^2}{\kappa^2}+\eta.
    \end{equation*}
    On the other hand, since  \eqref{eq: consequence of uniqueness of minimizer} continues to hold in the case $\eps=0$ with $\eta=0$, we have that for any $a\ne \hat a$, \begin{equation*}
        S_{a}(v_{a,\square})\ge -\frac{\beta^2}{\kappa^2}.
    \end{equation*}
    Combining the last two displays with \eqref{eq: equivalent objective function}, we arrive at \begin{align*}
        D(r\|\kappa^{-2}\1\1^\intercal)- \beta^2\|r\|_F^2 = \frac{1}{\kappa}\Bigl( S_{\hat a}(v_{\hat a,\square})+ \sum_{a\ne \hat a}S_a(v_{a,\square})\Bigr)\ge  -\frac{\beta^2}{\kappa^2}+\frac{\eta}{\kappa}=  -\beta^2\|\kappa^{-2}\1\1^\intercal\|_F^2+\frac{\eta}{\kappa},
    \end{align*}
      which together with \eqref{eq: miracle of uniform} yields the desired result.
\end{proof}
\begin{remark}
    Without resorting to the row decomposition \eqref{eq: equivalent objective function}, one may directly apply \cite[Theorem 2.1]{EW90} by ``flattening''  a  matrix $r$ to a $\kappa^2$-dimensional vector.
    However, this will yield an upper bound on  $\beta$ with worse dependence on $\kappa$.
\end{remark}

\begin{proof}[\bf Proof of Lemma~\ref{lem: local KL divergence}]
    For convenience, define $x=p-q $. 
    From the assumption $\max_{i\in[n]}|p-q| \le 2^{-1}\min_{i\in[n]}q_i$,   \begin{equation}\label{eq: bound x_i by q_i}
        \max_{i\in[n]}|q_i^{-1}x_i|\le 2^{-1}.
    \end{equation}
    Then,   Taylor's expansion $|\log (1+a)-a +2^{-1}a^2|\le 3|a|^3$ for $|a|\le 2^{-1}$ gives \begin{align*}
        D(p\|q) &= \sum_{i\in[n]} (q_i+x_i)\log (1+ q^{-1}_ix_i) =  \sum_{i\in[n]}(q_i+x_i) ( q^{-1}_ix_i -2^{-1} q^{-2}_ix_i^2+r_i)
        \\&= 2^{-1}\sum_{i\in[n]}q_i^{-1}x_i^2 -2^{-1}\sum_{i\in [n]}q_i^{-2}x_i^3+ \sum_{i\in[n]}(q_i+x_i)r_i
    \end{align*} for some $|r_i|\le 3 q_i^{-3}|x_i|^3$. 
    Here, we used  $\sum_{i\in [n]}x_i= 0$ in the last equality.
    The proof is now finished by the following control of the error terms, \[\Bigl| -2^{-1}\sum_{i\in [n]}q_i^{-2}x_i^3+ \sum_{i\in[n]}(q_i+x_i)r_i\Bigr|\le \frac{7}{2}\sum_{i\in[n]}q_i^{-2}|x_i|^3+3\sum_{i\in[n]}q_i^{-3}|x_i|^4 \le 5 (\min_{i\in[n]}q_i)^{-2}\sum_{i\in[n]}|x_i|^3, \] where we used \eqref{eq: bound x_i by q_i} for the last inequality.
    
\end{proof}

\section{Color symmetry breaking at zero temperature}\label{sec: color symmetry breaking at zero temp}
In this section, we verify color symmetry breaking at zero temperature for $\kappa \ge 56$, essentially by refining Mourrat's approach in \cite{Mou25}.
From the  inequalities $
    \GSE^\bal_{N,\kappa} \le \beta^{-1}F^\bal_{N,\beta}\le \GSE^\bal_{N,\kappa}+ \beta^{-1}\log \kappa$, we have \[\lim_{N\to\infty}\GSE^\bal_{N,\kappa} = \inf_{\beta>0}\frac{1}{\beta}\lim_{N\to\infty}F^\bal_{N,\beta}.\]
Since  the right-hand side in Theorem~\ref{thm: high temp} is an upper bound of $\lim_{N\to\infty}F^\bal_{N,\beta}$ for any $\beta>0$ per the Parisi formula \cite{BS24}, the preceding display implies \begin{equation}\label{eq: upper bound of balanced GSE}
    \lim_{N\to\infty}\GSE^\bal_{N,\kappa} = \inf_{\beta>0}\frac{1}{\beta}\lim_{N\to\infty}F^\bal_{N,\beta}\le  \inf_{\beta>0}\Bigl(\frac{\log \kappa }{\beta}+ \frac{\beta (\kappa-1)}{2\kappa^2}\Bigr)=  \Bigl( \frac{2(\kappa-1)\log \kappa}{\kappa^2}\Bigr)^{1/2}.
\end{equation}
We remark that this bound can also be obtained without the Parisi formula, by applying a well-known upper bound of the maximum of Gaussian processes as in \cite[Proposition 1.1.3]{Tal03_book} to the centered Hamiltonian \eqref{eq: centered Hamiltonian}.
Now, recall the lower bound \cite[Inequality (0.6)]{Mou25} for the limiting unconstrained ground state energy, \begin{equation*}
   \frac{2}{3\sqrt{\pi}}\le \lim_{N\to\infty}\GSE_{N, \kappa}.
\end{equation*}
Combining the last two displays, we obtain color symmetry breaking at zero temperature whenever \begin{equation*}
    \Bigl(\frac{2(\kappa-1)\log \kappa}{\kappa^2}\Bigr)^{1/2} <\frac{2}{3\sqrt{\pi}}.
\end{equation*} 
One can numerically check that $\kappa \ge 56$ guarantees the last inequality and, hence, color symmetry breaking at zero temperature.

\bibliographystyle{acm}
\setlength{\bibsep}{0.5pt}   

{\footnotesize\bibliography{references}}

\end{document}